\documentclass[a4paper,oneside]{article}

\usepackage{amsmath}
\usepackage{amssymb}
\usepackage[latin1]{inputenc}
\usepackage[numbers]{natbib}

\usepackage{geometry}
\newtheorem{theorem}{Theorem}[section]

\newtheorem{corollary}[theorem]{Corollary}

\newtheorem{definition}[theorem]{Definition}

\newtheorem{lemma}[theorem]{Lemma}

\newtheorem{proposition}[theorem]{Proposition}
\newtheorem{remark}[theorem]{Remark}

\newenvironment{proof}[1][Proof]{\textbf{#1.} }{\ \rule{0.5em}{0.5em}}

\newcommand{\refeqn}[1]{(\ref{#1})}
\newcommand{\virgolette}[1]{``#1''}
\newcommand{\cinf}[0]{C^{\infty}}

\newcommand{\matr}[0]{\operatorname{Mat}}
\newcommand{\spann}[0]{\operatorname{span}}

\begin{document}

\title{Reduction and reconstruction of stochastic differential equations via symmetries }
\author{Francesco C. De Vecchi\thanks{Dip. di Matematica, Universit\`a degli Studi di Milano, via Saldini 50, Milano, \emph{email: francesco.devecchi@unimi.it}} , Paola Morando\thanks{DISAA, Universit\`a degli Studi di Milano, via Celoria 2, Milano, \emph{email: paola.morando@unimi.it}}  and Stefania Ugolini\thanks{Dip. di Matematica, Universit\`a degli Studi di Milano, via Saldini 50, Milano, \emph{email: stefania.ugolini@unimi.it}}}


\date{}

\maketitle
\abstract{
An algorithmic method to exploit a general class of infinitesimal symmetries for reducing  stochastic differential equations is presented and
 a natural definition of reconstruction,  inspired by the classical reconstruction by quadratures, is proposed.  As a side result
 the well-known  solution formula for linear one-dimensional stochastic differential equations is obtained within this symmetry approach. The complete procedure is applied to several examples
 with both theoretical and  applied relevance.}

\bigskip



\section{Introduction}

Stochastic differential equations (SDEs) and, more in general, stochastic processes provide a very useful tool for  the description of many phenomena in physics,
meteorology, biology, social sciences, economics or finance.
For this reason there is a great interest in finding new simple and flexible models  with some kind of  analytical tractability. Just to mention some recent examples,
in \cite{Craddock2012,Platen2010} the Authors focus on models with closed form of the marginal probability density, while in \cite{polynomial_processes} systems with
explicit formulas for the moments are discussed. Moreover,
in   \cite{Craddock2015,Cordoni2014,affine},  systems with exact expression for the Fourier (or generalized Laplace) transform of some significant functionals of the
process are considered. \\
These examples also suggest that the knowledge of closed-form expressions   for some mathematical objects related with  SDEs can be useful in order to formulate faster
or more stable algorithms for numerical simulation (see e.g. \cite{Bormetti2014,Heston1993,Malham_Wiese2008}), to propose better estimators for statistical inference
(see e.g. \cite{Affine_parameter_estimation,Belomestny_Reiss2006,Forman2008}) or to reduce the complexity of the  models using asymptotic expansions or perturbation
theory techniques (see e.g. \cite{SABR2,Lorgin_Pagliarani_Pascucci2015}).\\ 
For (deterministic) ordinary differential equations (ODEs) a powerful method to select  general systems admitting closed analytical formulas for their solutions
is the infinitesimal symmetries method originally  proposed by Sophus Lie and thereafter  developed by many Authors (see, e.g. \cite{Olver,Stephani} and references therein).
 The underlying idea is to provide  an algorithmic procedure in order  to identify symmetric ODEs and to exploit  their symmetries for simplifying  them.\\
Despite the well acknowledged and rich literature in the deterministic setting, the concept of infinitesimal symmetries for SDEs is quite recent
\cite{DMU,Gaeta1,Lazaro_Ortega,Lescot,Grigoriev,Zung} and their use for reduction purposes  is not yet completely developed.\\
The principal aim of this paper is to investigate the possible applications of a symmetry approach to  SDEs  taking the cue from  the deterministic case.
 We   propose a simple algorithmic method to exploit  infinitesimal symmetries for reducing a SDE and,  in the particular case of SDEs admitting a solvable
  Lie Algebra of symmetries,  we provide a reconstruction procedure allowing to obtain the  solution to the original SDE starting from the solution to the reduced one.

Although these basic applications of infinitesimal symmetries have already been discussed in  \cite{Kozlov2010,Lazaro_Ortega,Zambrini2010,Zung},
there are several novelties in our approach. First of all we use the weaker notion of infinitesimal symmetry of a SDE proposed in \cite{DMU}, including both
 \emph{strong symmetries}, used for reduction and reconstruction in \cite{Lazaro_Ortega}, and \emph{quasi-strong symmetries} used  for reduction in \cite{Zung}.
  It is worth to remark that our  notion of symmetry  allows us to take into account also random time change transformations which have not been considered by previous
  Authors. Moreover, the choice of a SDE formulation to deal  with  symmetries is more convenient than a formulation based on the generator operator of the SDE.  Indeed, even though the notion of symmetry    based on
   the generator is quite effective for reduction purposes, it turns out to be inadequate when one aims at generalizing to the stochastic framework  the standard reconstruction by quadratures. \\
A further advantage of our method, inspired by the original ideas of Lie, is that we do not need the existence of a Lie group action related to the infinitesimal
symmetries as required in \cite{Lazaro_Ortega,Zung}. The possibility of working    with both the global and the local  action of a Lie group turns out to be very useful in
 order to deal with   stochastically complete SDEs admitting  infinitesimal symmetries which do not generate a globally defined flow of  diffeomorphisms (see the example of Subsection \ref{subsection_singular}).\\
Furthermore we propose a notion of reconstruction inspired by the classical idea of reconstruction by quadratures and similar to the one proposed in \cite{Kozlov2010} for strong symmetries. We remark that  this concept  is different from the one proposed in \cite{Lazaro_Ortega} and  our corresponding  natural notion of integrability, which is in some respects more restrictive than that of \cite{Zung} by not including higher order symmetries, allows us to exploit (not only Abelian but) general solvable algebras of  symmetries. Besides,  we recall that the stochastic quadrature procedure for one-dimensional diffusion processes proposed in \cite{Zambrini2010} cannot be directly related to our results, since it is based on a well-defined variational structure. We discuss some particular examples of SDEs with a variational structure in Subsection \ref{subsection_mechanics}.\\
It is important to note that our approach is  completely explicit and allows us to compute symmetries of a SDE  by solving an  overdetermined
system of first order PDEs.   In particular we apply our complete procedure to
a class of one-dimensional diffusions reducing to linear SDEs for a particular choice of the parameters. In this case,
considering a suitable two-dimensional SDE including the original one, we are able to find the  explicit solutions recovering
the well known solution formula for one-dimensional linear SDEs,
together with the usual
change of variables coupled with the associated  homogeneous equation.

Moreover we consider a class of  (stochastic) mechanical models which includes the standard perturbations of stochastic Lagrangian systems.
In particular, starting from a mechanical system describing
a particle subjected to  forces  depending on the velocities, we look for general stochastic perturbations of the deterministic system  preserving the symmetries
and we analyze in details a  couple of significant
examples within this  class.\\
Finally, we apply our symmetry approach to prove the integrability of a well-known financial stochastic model (SABR) and we discuss the possible generalizations of our results suggested by this analysis.  \\
The paper is organized  as follows. In Section \ref{section_preliminary}, for the convenience of the reader, we collect some general facts about foliations and reduction maps. In Section \ref{section_weak_symmetry} we introduce the definition of infinitesimal symmetry of a SDE as proposed in \cite{DMU} and we recall some results and formulas we need in  the rest of the paper. Reduction, reconstruction and integrability of a SDE are described in  Section \ref{section_reduction_reconstruction} and in Section \ref{section_examples} the general  theory is applied  to some explicit relevant  examples in order to point out the effectiveness of the proposed method.

\section{Some geometric preliminaries}\label{section_preliminary}

In this section we recall some geometric preliminaries  needed in the  following. In particular in Subsection \ref{subsectionFoliations and projections} we introduce a class of foliations  which turns out to be useful in order to reduce SDEs, whereas in Subsection \ref{subsection_solvable_coordinate}  we describe adapted coordinate systems for  solvable Lie algebras of vector fields which  are exploited in the reconstruction process.\\
We work  in an open set $M$ of the Euclidean space $\mathbb{R}^n$ with a cartesian coordinate system $x^i$  and  we  denote by $U(x_0)$   an open neighborhood of a point $x_0 \in M$. All (real or matrix-valued) functions on $M$ are smooth and  we denote by $\partial_i$ the partial derivatives with respect to $x^i$. If $A:M \rightarrow \matr(k,m)$, we write $A^l_r$ for the $l$-th row and $r$-th column component of the matrix $A$,  identifying $\matr(k,1)$ with $\mathbb{R}^k$. \\
Given a function $\Psi:M \rightarrow \mathbb{R}^h$  and a vector field $Y\in TM$, the push forward of $Y$ by $\Psi$ is defined by
$$\Psi_*(Y)=\nabla(\Psi) \cdot Y,$$
where
$$(\nabla (\Psi))^i_j=\partial_j(\Psi^i), \qquad i=1,...,h \qquad j=1,...,n$$
and $\cdot$ denotes the usual matrix product.
If $A:M \rightarrow \matr(h,k)$ we define $Y(A):M \rightarrow \matr(h,k)$ as
$$(Y(A))^i_j=Y(A^i_j),$$
where $ Y(A^i_j)$ is the usual notation for the vector field $Y$ applied to the real-valued functions $A^i_j$.\\
In the following we denote by $I_{n}$ the $n$-dimensional identity matrix and by $SO(m)$ the group of the orthogonal matrices.

\subsection{Foliations and projections}\label{subsectionFoliations and projections}

 Let  $Y_1,...,Y_k$  be  a  set of vector fields on $M$ such that the  distribution $\Delta =\spann \{Y_1,...,Y_k\}$ is of  constant rank $r$.
 If  $\Delta$ is integrable, i.e. $[Y_i,Y_j] \in \Delta$ for every $i,j=1,...,k$, then $\Delta$ defines a foliation on $M$. Moreover, if
there is a submersion $\Psi:M \rightarrow M'$, where (possibly  restricting $M$) $M'$ is an open subset of $\mathbb{R}^{n-r}$ such that
$$\Delta=\ker(\nabla(\Psi))$$
and the level sets of $\Psi$ are connected subsets of $M$, the foliation defined  by  $\Delta$ can be used for   reduction purposes. In fact, under these assumptions,  $\Psi$ is a surjective submersion and  the level sets of $\Psi$ are connected closed submanifolds of $M$.
\begin{definition}\label{definition_reduction_symmetries}
A surjective submersion $\Psi:M \rightarrow M'$ is  a \emph{reduction map} if $ \Psi$  has connected level sets.
The vector fields $Y_1,...,Y_k$ generating an integrable distribution $\Delta$ of constant rank $r$ are   \emph{reduction vector fields} for the reduction map $\Psi:M \rightarrow M'$ if
$$\spann\{Y_1(x),...,Y_k(x)\}=\ker \nabla(\Psi)(x) \qquad  \forall x \in M.$$
\end{definition}
We remark that, if $Y_1,...,Y_k$ are reduction vector fields for the reduction map $\Psi$, then $(M,\Psi,M')$ is a fibred manifold.\\
\begin{definition}\label{definition_regular_vectors}
A set of vector fields  $Y_1,...,Y_r$ on $M$ is  \emph{regular} on $M$ if, for any $x \in M$, the vectors $Y_1(x),...,Y_r(x)$ are linearly independent.
\end{definition}
A set of vector fields $Y_1,...,Y_k$ generating an integrable distribution $\Delta$ of constant rank does not admit in general a global reduction map, but the following well known local result holds.
\begin{proposition}[Frobenius theorem]\label{proposition_reduction}
Let $Y_1,...,Y_k$ be a set of vector fields generating a regular integrable distribution $\Delta$ of constant rank  $r$. Then, for any $x \in M$, there exist a neighborhood $U$ of $x$ and a reduction map $\Psi:U \rightarrow U' \subset \mathbb{R}^{n-r}$ such that $Y_1,...,Y_k$ are reduction vector fields for $\Psi$.
\end{proposition}
We remark that  the classical  reduction of a manifold under a Lie group  action is included in Definition \ref{definition_reduction_symmetries}.
 Indeed, given a connected Lie group $G$ acting on $M$,  we can naturally  define  an equivalence relation and the quotient manifold $M'=M/G:=M/\sim$. If the action of $G$ is proper and free, $M'$ admits a natural structure of $(n-r)$-dimensional manifold (see \cite{Olver}) and the natural projection $\Pi:M \rightarrow M'$ is a submersion. Moreover,  if $G$ is connected and $\{Y_1,...,Y_r\}$ are the generators of the corresponding Lie algebra,  $\Pi$ is a reduction map and
$$\spann\{Y_1(x),...,Y_r(x)\}=\ker(\nabla(\Pi)(x)) \qquad \forall x \in M.$$
In this case $(M,\Psi,M')$ is not only a fibred manifold but also a principal bundle with structure group $G$.

\begin{proposition}\label{lemma_submersion}
Let $\Psi:M \rightarrow M'$ be a reduction map and suppose that the vector fields $\{Y_1,...,Y_k\}$ are reduction vector fields for $\Psi$.
If $M$ is connected,  for any function $f \in \cinf(M)$ such that $Y_i(f)=0$ there exists a unique function $f' \in \cinf(M')$ such that
$$f=f' \circ \Psi.$$
Moreover, if $\mathcal{G}=\spann\{Y_1,...,Y_k\}$ and  $Y$ is a vector field on $M$ such that
$$[Y,\mathcal{G}] \subset \mathcal{G},$$
there exists an unique vector field $Y'$ such that
$$\Psi_*(Y)= \nabla (\Psi) \cdot Y= Y' \circ \Psi.$$
\end{proposition}
\begin{proof}
See Chapter 4 of \cite{Moerdijk}.
${}\hfill$ \end{proof}

\subsection{Solvable algebras and adapted coordinate systems }\label{subsection_solvable_coordinate}

For later use, in this subsection we discuss the local existence of a suitable  adapted coordinate system on $M$ such that the generators $Y_1,...,Y_r$ of a solvable Lie algebra have a special form.

\begin{definition}\label{definition_solvable_coordinate}
Let $Y_1,...,Y_r$ be a set of regular vector fields on $M$ which are generators of a solvable Lie algebra $\mathcal{G}$. We say that $Y_1,...,Y_r$ are in \emph{canonical  form} if there are $i_1,...,i_l$ such that $i_1+...+i_l=r$ and
$$(Y_1|...|Y_r)=\left(\begin{array}{c|c|c|c}
I_{i_1} & G^1_1(x) & ... & G^1_l(x) \\
\hline
0 & I_{i_2} & ... & G^2_l(x)\\
\hline
\vdots & \ddots & \ddots & \vdots \\
0 & 0 & ... & I_{i_l}\\
\hline
 0 & 0 & 0 & 0 \end{array} \right), $$
where  $G^h_k:M \rightarrow \matr(i_h,i_k)$ are smooth functions.
\end{definition}

\begin{theorem}\label{theorem_solvable_coordinate}
Let $\mathcal{G}$ be an $r$-dimensional solvable Lie algebra on $M$ such that $\mathcal{G}$ has constant dimension $r$ as a  distribution of $TM$ and let $\Psi$ be a reduction map for $\mathcal{G}$. Then, for any  $x_0 \in M$, there is a set of generators $Y_1,...,Y_r$  of $\mathcal{G}$ and a local diffeomorphism $\Phi:U(x_0) \rightarrow \tilde{M}$ of the form
\begin{equation}\label{equationPhi}
\Phi=\left( \begin{array}{c}
\tilde{\Phi}\\
\Psi \end{array}\right),
\end{equation}
such that $\Phi_*(Y_1),...,\Phi_*(Y_r)$ are generators in canonical form for $\Phi_*(\mathcal{G})$.
\end{theorem}
\begin{proof}
Since $\mathcal{G}$ is solvable, denoting by  $\mathcal{G}^{(0)}=\mathcal{G}$ and $\mathcal{G}^{(i+1)}=[\mathcal{G}^{(i)},\mathcal{G}^{(i)}]$, there exists  $l\geq 0$ such that $\mathcal{G}^{(l)}\not =0$ and $\mathcal{G}^{(l+1)}=\{0\}$. Let  $Y_1,...,Y_{i_1}$ be the  generators of $\mathcal{G}^{(l)}$, $ \ Y_{i_1+1},...,Y_{i_1+i_2}$ be the  generators of $\mathcal{G}^{(l-1)} \setminus \mathcal{G}^{(l)}$ and, in general, $Y_{i_1+...+i_{k-1}+1},...,Y_{i_1+...+i_k}$ be the generators of $\mathcal{G}^{(l-k+1)}\setminus \mathcal{G}^{(l-k)}$.
Since $(M,\Psi, M')$ is a fibred manifold, for any  $x_0 \in M$ we can consider a local  smooth section $S:V(\Psi(x_0)) \rightarrow M$
defined in $V(\Psi(x_0))$ and we can construct a local diffeomorphism on  $W \times V$ (where $W \subset \mathbb{R}^r$)  transporting $S(x_0)$ along the flows $\Phi^i_{a_i}$ of the vector fields  $Y_i$. In particular, considering the function $F:W \times V \rightarrow M$ (where $W$ is a neighborhood of $0$ in $\mathbb{R}^r$) defined by
$$F(a_1,...,a_r,x'^1,...x'^{n-r})=\Phi^1_{a_1}(...(\Phi^r_{a_r}(S(x'^1,...,x'^r)))...)$$
we can define  $\Phi=F^{-1}$. Indeed it is easy to prove that $F$ is a local diffeomorphism since $Y_1,...,Y_r$ form a regular set of vector fields, $S$ is a local section of the foliation  $(M,\Psi,M')$ and $\Psi$ is a reduction function for  $Y_1,...,Y_r$. Furthermore, since $F$ is obtained by composing the flows of $Y_1,...,Y_r$ in the natural order (i.e. respecting the solvable structure of $\mathcal{G}$), it is easy to prove that $\Phi_*(Y_1),...,\Phi_*(Y_r)$ are in canonical form.

${}\hfill$ \end{proof}

\begin{remark} In the particular case of a solvable connected Lie group $G$ acting  freely and regularly on $M$, Theorem \ref{theorem_solvable_coordinate} admits a global version. Indeed, under these hypotheses,  $\Phi$ can be defined in an open set $U$ of the form $U=\Psi^{-1}(V)$ where $V$ is an open set of $M'$ and $(M,\Psi, M')$ turns out to be a principal bundle with structure group $G$. So for a neighborhood $V$ of $\Psi(x_0)$ the set $U=\Psi^{-1}(V)$ is diffeomorphic to $V \times G$ and the generators $Y_1,...,Y_r$ of $G$ are vertical vector fields with respect to the bundle structure of $M$. Furthermore, it is possible to choose  a global coordinate system  $g^1,...,g^r$ on $G$ such that $Y_1,...,Y_r$ are in canonical form (see for example \cite{Onishchik}, Chapter 2, Section 3.1, Corollary 1) and equation \eqref{equationPhi} is given by  $\Phi=(g^1,...,g^r,\Psi^1,...,\Psi^{n-r})^T$. Obviously, if $(M,\Psi,M')$ is a trivial bundle, the diffeomorphism $\Phi$ of Theorem \ref{theorem_solvable_coordinate} can be defined globally.

\end{remark}

\section{Stochastic differential equations and their  symmetries}\label{section_weak_symmetry}

\subsection{Stochastic differential equations}

In the following,   given a filtration $\mathcal{F}_t \subset \mathcal{F} $, we consider only stochastic processes $(\Omega,\mathcal{F}, \mathcal{F}_t, \mathbb{P})$  which are  adapted with the  filtration $\mathcal{F}_t $.\\

\begin{definition}
 Let $\mu:M
\rightarrow \mathbb{R}^n$ and $\sigma:M \rightarrow \matr(n,m)$ be two smooth functions. A stochastic process $X$ on $M$ and an $m$-dimensional Brownian motion $W$ (in short the process $(X,W)$) solve (in the weak sense) the SDE with coefficients
$\mu,\sigma$ until the stopping time
$\tau$ (or  shortly solves the SDE $(\mu,\sigma)$) if for
any $t \in \mathbb{R}_+$
$$X^i_{t \wedge \tau}-X^i_0=\int_0^{t \wedge \tau}{\mu^i(X_s)ds}+\int_0^{t \wedge \tau}{\sum_{\alpha=1}^m\sigma_{\alpha}^i(X_s)dW^{\alpha}_s}.$$
If the process  $(X,W)$ solves the SDE $(\mu,\sigma)$ we write, as usual,
\begin{eqnarray*}
dX_t&=&\mu(X_t)dt+\sigma(X_t)\cdot dW_t\\
&=&\mu dt+\sigma \cdot dW_t.
\end{eqnarray*}
\end{definition}
The stopping time $\tau$ is supposed to be strictly less then the explosion time of the SDE. \\
It is well known that with a SDE $(\mu,\sigma)$  it is possible to
associate a second order differential operator
$$L=\sum_{\alpha=1}^m \sum_{i,j=1}^n\frac{1}{2}\sigma^i_{\alpha}\sigma^j_{\alpha}\partial_{ij}+\sum_{i=1}^n \mu^i \partial_i.$$
The operator $L$ is called the infinitesimal
generator of the process and appears, for example, in the
following  important formula.
\begin{proposition}[Ito formula]\label{Itoformula}
Let $(X,W)$ be a solution to  the SDE $(\mu,\sigma)$ and let $f:M \rightarrow
\mathbb{R}$ be a smooth function. Then  $F=f(X)$ satisfies
$$dF_t=L(f)(X_t)dt+\nabla(f)(X_t) \cdot \sigma(X_t) \cdot dW_t.$$
\end{proposition}

\subsection{Stochastic transformations and  symmetries of SDEs}\label{subsection_weak_symmetry}

In this section we briefly recall the general definitions of finite and infinitesimal stochastic transformations of a SDE introduced  in \cite{DMU}, in order to provide a  general definition of symmetry, weak enough to include interesting examples in the stochastic framework.

\begin{definition}\label{stochastic_transformation}

Given a diffeomorphism $\Phi:M \rightarrow M'$ and two smooth functions  $B:M \rightarrow SO(m)$ and  $\eta:M \rightarrow \mathbb{R}_+$, the triad $T=(\Phi,B,\eta)$ is called a \emph{stochastic transformation} between $M$ and $M'$. Stochastic transformations of the form $(\Phi,I_m,1)$  are called \emph{strong}  whereas stochastic transformations of the form $(\Phi,B,1)$  are called \emph{quasi-strong}.
\end{definition}
Hereafter we identify the strong stochastic transformations $T=(\Phi,I_m,1)$  with the diffeomorphisms $\Phi$ and we denote by $S_m(M,M')$ the set of all stochastic transformations between $M$ and $M'$. \\
A stochastic transformation $T=(\Phi,B,\eta)$ naturally defines two transformations on  processes  and  SDEs.
Indeed, if $H_{\eta}$ denotes  the random time change given by the expression
$$t'=\int_0^t{\eta(X_s)ds}$$
and $(X,W)$ is
a process,   we define $P_T(X,W)=(X',W')$ where
\begin{eqnarray*}
X'=\Phi(H_{\eta}(X))\\
W'=H_{\eta}(\tilde{W}).
\end{eqnarray*}
and
$$\tilde{W}^{\alpha}_t=\int_0^t{\sum_{\beta=1}^m B^{\alpha}_{\beta}(X_s) dW^{\beta}_s}.$$
Moreover, given a SDE $(\mu, \sigma)$, we define $E_T(\mu,\sigma)=(\mu',\sigma')$ where
\begin{eqnarray*}
\mu'&=&\left(\frac{1}{\eta}L(\Phi)\right) \circ \Phi^{-1}\\
\sigma'&=&\left(\frac{1}{\sqrt{\eta}} \nabla(\Phi) \cdot \sigma \cdot B^T \right) \circ \Phi^{-1}.
\end{eqnarray*}
With these definitions it is easy to prove that if $(X,W)$ is a solution to $(\mu,\sigma)$, then $P_T(X,W)$ is a solution to $E_T(\mu,\sigma)$.\\
In \cite{DMU} we prove that the set $S_m(M)=S_m(M,M)$ is a group with  composition law defined by
$$T' \circ T=(\Phi' \circ \Phi,(B'\circ \Phi) \cdot B, (\eta' \circ \Phi) \eta)$$
and unit $1_M=(id_M,I_m, 1)$. Moreover, if we consider the  one-parameter group $T_a$ such that $T_a \circ T_b=T_{a+b}$ and $T_0=1_M$, there exist a vector field $Y$ and two smooth functions $C:M \rightarrow so(m)$ and $\tau:M \rightarrow \mathbb{R}$ such that \begin{equation}\label{equation_infinitesimal_SDE2}\begin{array}{lcr}
\partial_a(\Phi_a(x))&=&Y(\Phi_a(x))\\
\partial_a(B_a(x))&=&C(\Phi_a(x)) \cdot B_a(x)\\
\partial_a(\eta_a(x))&=&\tau(\Phi_a(x))\eta_a(x).
\end{array}
\end{equation}
\begin{definition}
 The triad $V=(Y,C,\tau)$ defined  by \eqref{equation_infinitesimal_SDE2} is called \emph{infinitesimal stochastic transformation} and the set of these triads is denoted by $V_m(M)$. The infinitesimal stochastic transformations of the form $(Y,0,0)$ are called \emph{strong},  whereas infinitesimal stochastic transformations of the form $(Y,C,0)$  are called \emph{quasi-strong}.
\end{definition}
In the following  we identify strong  infinitesimal stochastic transformations $V=(Y,0,0)$  with the vector fields $Y$.\\
The set of infinitesimal stochastic transformations turns out to be a Lie algebra with respect to the
 Lie bracket
\begin{eqnarray*}
[V_1,V_2]&=&([Y_1,Y_2],Y_1(C_2)-Y_2(C_1)-\{C_1,C_2\},Y_1(\tau_2)-Y_2(\tau_1)),
\end{eqnarray*}
where $\{C_1,C_2\}$ denotes the usual commutator between matrices.
Moreover the action of a finite stochastic transformation on an infinitesimal one  takes the form
\begin{eqnarray*}
T_*(V)&=&((\nabla(\Phi) \cdot Y) \circ \Phi ^{-1} ,(B \cdot C \cdot B^{-1}+Y(B) \cdot B^{-1}) \circ \Phi^{-1},\\
&&(\tau+Y(\eta)\eta^{-1})\circ \Phi^{-1}).
\end{eqnarray*}
The important role of strong infinitesimal stochastic transformations in reduction of SDEs makes the following theorem very useful for  applications.
\begin{theorem}\label{theorem_infinitesimal_SDE1}
 Let $K=\spann \{V_1,...,V_k\}$ be  a Lie algebra of infinitesimal stochastic transformations and  $V_i=(Y_i,C_i,\tau_i)$. If  $ x_0 \in M$  is such that $Y_1(x_0),...,Y_k(x_0)$ are linearly independent, there exist an
open neighborhood $U(x_0)$  and a stochastic transformation $T=(id_U,B,\eta) \in S_m(U)$
such that
$T_*(V_1),...,T_*(V_k)$ are strong infinitesimal stochastic
transformations. Furthermore the smooth functions
$B,\eta$ are solutions to the equations
\begin{eqnarray}
Y_i(B)&=&-B \cdot C_i \label{equation_weak_to_strong1}\\
Y_i(\eta)&=&-\tau_i \eta, \label{equation_weak_to_strong2}
\end{eqnarray}
for $i=1,...,k$.
\end{theorem}
\begin{proof}
See \cite{DMU}.
${}\hfill$ \end{proof}

In order to make the  definition of a symmetry of a SDE  more readable, we introduce the following notation: given $H \in TM$ and $K:M \rightarrow \matr(n,r)$,  we define the matrix-valued function $E:=\left[H,K\right]:M \rightarrow \matr(n,r)$ by
$$E^i_j=\sum_{k=1}^n(H^k\partial_k(K^i_j)-K^k_j \partial_k(H^i)).$$
It is easy to prove that the brackets $\left[\cdot,\cdot\right]$ generalize the standard Lie brackets for vector fields.
\begin{definition}
A stochastic transformation  $T$ is a \emph{(finite) symmetry} of a SDE $(\mu,\sigma)$ if
$$E_T(\mu,\sigma)=(\mu,\sigma).$$
An infinitesimal stochastic transformation is an \emph{(infinitesimal) symmetry } of the SDE $(\mu,\sigma)$ if the following \emph{determining equations} hold
\begin{eqnarray}\label{determining_eq}
[Y,\sigma]&=&-\frac{1}{2} \tau \sigma-\sigma \cdot C\\
Y(\mu)-L(Y)&=&-\tau \mu.
\end{eqnarray}
\end{definition}
It is easy to prove that  $T \in S_m(M)$ is a symmetry of the SDE $(\mu,\sigma)$ if and only if, for any solution $(X,W)$ to  $(\mu,\sigma)$,  also $P_T(X,W)$ is a solution to  $(\mu,\sigma)$. Furthermore if $V \in V_m(M)$ generates a one parameter group $T_a$, $V$ is an (infinitesimal) symmetry of the SDE $(\mu,\sigma)$ if and only if $T_a$ is a (finite) symmetry of  $(\mu,\sigma)$. Moreover the following result holds.
\begin{theorem}
Let  $V \in V_m(M)$ be a symmetry of the SDE $(\mu,\sigma)$. If $T \in S_m(M)$,  then $T_*(V)$  is a symmetry of $E_T(\mu,\sigma)$.
\end{theorem}

\begin{corollary}\label{corollary_strong_symmetry}
Let $V_1=(Y_1,C_1,\tau_1),...,V_k=(Y_k,C_k,\tau_k)$ be infinitesimal symmetries of $(\mu,\sigma)$.
 If $x_0 \in M$ is such that $Y_1(x_0),...,Y_k(x_0)$ are linearly independent,  then there exist a neighborhood $U(x_0)$ and a stochastic transformation $T \in S_m(U,U')$ such that $T_*(V_i)$ are strong infinitesimal symmetries of $E_T(\mu,\sigma)$.
\end{corollary}

\begin{remark}
The concept of strong symmetry has been introduced, with the simple name of symmetry, in \cite{Lazaro_Ortega}, where the noise is not necessarily a Brownian motion but any continuous semimartingale.  It is easy to prove that the results of this section hold if  we replace Brownian motion with any continuous semimartingale. Anyway in the following we restrict to  Brownian motion since  the weaker notion of   symmetry in \cite{DMU} has been introduced  only for  SDEs driven by Brownian motion.
\end{remark}

\section{Reduction and reconstruction }\label{section_reduction_reconstruction}

In this section we propose a generalization of some  well known  results of symmetry reduction for ODEs to the stochastic framework.  Moreover we provide suitable  conditions for a symmetry to be inherited by the reduced equation and we tackle the problem of the reconstruction of the solution to the original SDE starting from the knowledge of the solution to the reduced one.

\subsection{Reduction}

\begin{theorem}\label{theorem_reduction1}
Let $\{Y_1,...,Y_k\}$ be a set of reduction vector fields for the reduction map $\Psi:M \to M'$ such that $(Y_1,C_1,0),...,(Y_k,C_k,0)$ are  quasi-strong symmetries of the SDE $(\mu,\sigma)$. If $\nabla(\Psi) \cdot \sigma \cdot C_i =0$ ($\forall i=1,...,k$),
 there exists an unique SDE $(\mu',\sigma')$ on $M'$ such that
\begin{eqnarray*}
L(\Psi)&=&\mu'\circ \Psi\\
\nabla(\Psi) \cdot \sigma &=&\sigma' \circ \Psi.
\end{eqnarray*}
Furthermore if $(X,W)$ is a solution to  $(\mu,\sigma)$, then $(\Psi(X),W)$ is a solution to  $(\mu',\sigma')$.
\end{theorem}
Before proving Theorem \ref{theorem_reduction1} we introduce the following lemma.
\begin{lemma}\label{lemma_quasi_strong}
If $(Y,C,\tau)$ is  an infinitesimal   symmetry of the SDE $(\mu,\sigma)$, for any smooth function $f$ we  have
\begin{eqnarray*}
\nabla(Y(f)) \cdot \sigma- Y(\nabla(f) \cdot \sigma)&=&\frac 12 \tau \nabla(f) \cdot \sigma+ \nabla(f) \cdot \sigma \cdot C.\\
L(Y(f))-Y(L(f))&=&\tau L(f).
\end{eqnarray*}
\end{lemma}
\begin{proof}
See \cite{DMU}.
${}\hfill$ \end{proof}

\medskip

\noindent\begin{proof}[Proof of Theorem \ref{theorem_reduction1}]
Since we are considering quasi-strong symmetries, Lemma \ref{lemma_quasi_strong} ensures that $Y_i(L(\Psi))=L(Y_i(\Psi))$ and, being $Y_i\in
\ker \nabla(\Psi)$, we have $Y_i(L(\Psi))=0$. Hence Proposition \ref{lemma_submersion} guarantees the existence of a function $\mu'$ such that $L(\Psi)=\mu' \circ \Psi$.\\
Moreover, in the case of quasi-strong symmetries, equation \eqref{determining_eq} reduces to $[Y_i,\sigma]=-\sigma\cdot C_i$ and the hypothesis
$\nabla(\Psi) \cdot \sigma \cdot C_i =0$ ensures that $[Y_i,\sigma]\in \ker \nabla(\Psi)$. Hence, denoting by $\sigma_{\alpha}$  the $\alpha$ column of $\sigma$, we have
$$[Y_i,\sigma_{\alpha}] \in \spann\{Y_1,...,Y_k\}.$$
Therefore, by Proposition \ref{lemma_submersion}, there exists an unique vector field $\sigma'_{\alpha}$ on $M'$ such that
$$\sigma'_{\alpha} \circ \Psi =\nabla(\Psi)\cdot  \sigma_{\alpha}$$
and, considering  the matrix-valued function $\sigma'$ with columns $\sigma'_{\alpha}$,  the theorem is proved.\\
${}\hfill$ \end{proof}

\begin{remark}
If $Y_1,...,Y_k$ are strong symmetries of the SDE $(\mu,\sigma)$, conditions $\nabla (\Psi) \cdot \sigma \cdot C_i=0$  of Theorem \ref{theorem_reduction1}
 are automatically  satisfied. However in Section \ref{section_examples} we provide interesting examples of SDEs admitting only quasi-strong symmetries.
 Indeed a consequence of Theorem 2.10 in \cite{Zung} is that if $V_1=(Y_1,C_1,0),...,V_k=(Y_k,C_k,0)$ are quasi-strong symmetries of
 $(\mu,\sigma)$ and $Y_1,...,Y_k$ generate an integrable distribution of constant rank,  there exists  a (local)  stochastic transformation
 $T=(id_M,B,1)$ such that $T_*(V_1),...,T_*(V_k)$ satisfy the hypotheses of Theorem \ref{theorem_reduction1} for the transformed SDE $E_T(\mu,\sigma)$.
 We remark that, if the fibred manifold $(M,\Psi,M')$ is not a trivial fibred manifold, $T$ is only locally defined.
\end{remark}

\begin{theorem}\label{reduction_theorem2}
In the hypotheses and with the notations of Theorem \ref{theorem_reduction1}, let  $V=(Y,C,\tau)$ be a symmetry of the SDE $(\mu,\sigma)$ such that, for any $i=1,...,k$,
\begin{equation}\label{pippo}
[Y,Y_i] \in \spann\{ Y_1, \ldots, Y_k\}, \qquad Y_i(C)=0, \qquad Y_i(\tau)=0.
\end{equation}
Then the infinitesimal transformation $(Y',C',\tau')$ on $M'$, where  $Y'=\Psi_*(Y)$, $C' \circ \Psi=C$ and $\tau' \circ \Psi=\tau$, is a symmetry
of the SDE $(\mu',\sigma')$.
\end{theorem}
\begin{proof}
We prove in detail that $Y'$ satisfies the determining equation \eqref{determining_eq} for $\sigma'$. The proof for $\mu'$ can be easily obtained
in a similar way.\\
Given  $f \in \cinf(M)$  such that $Y_i(f)=0$,   Proposition \ref{lemma_submersion} ensures that there exists a function $f' \in \cinf(M')$ such
that $f=f' \circ \Psi$. Moreover we have
\begin{equation}\label{equation_other_symmetry1}
\begin{array}{ccc}
Y(f)&=&(Y'(f')) \circ \Psi\\
\nabla(f) \cdot \sigma&=&(\nabla'(f') \cdot \sigma') \circ \Psi,
\end{array}
\end{equation}
where $\nabla'$ denotes the differential  with respect to  the  coordinates $x'^i$ of $M'$.\\
The determining equations \eqref{determining_eq} are equivalent to the relations
$$Y'(\sigma'^i)-\nabla'(Y'^i) \cdot \sigma'=-\frac{1}{2}\tau'\sigma'^i -\sigma'^i \cdot C' .$$
Since $(Y,C,\tau)$ is a symmetry of the SDE $(\mu,\sigma)$, by Lemma \ref{lemma_quasi_strong} for any  smooth function $f$ we have
\begin{equation}\label{equation_other_symmetry2}
Y(\nabla(f) \cdot \sigma)-\nabla(Y(f)) \cdot \sigma=-\frac{1}{2}\tau \nabla(f) \cdot \sigma-\nabla(f) \cdot \sigma \cdot C.
\end{equation}
Applying equations \refeqn{equation_other_symmetry1} and \refeqn{equation_other_symmetry2} we obtain
\begin{eqnarray*}
\left\{Y'(\sigma'^i)-\nabla'(Y'^i) \cdot \sigma'\right\} \circ \Psi&=&\left\{Y'(\nabla'(x'^i) \cdot \sigma')-\nabla'(Y'(x'^i)) \cdot \sigma'\right\} \circ \Psi\\
&=&Y((\nabla'(x'^i) \cdot \sigma') \circ \Psi)-\nabla(Y'(x'^i) \circ \Psi) \cdot \sigma\\
&=&Y(\nabla(\Psi^i) \cdot \sigma)-\nabla(Y(\Psi^i)) \cdot \sigma\\
&=&-\frac{1}{2} \tau (\nabla(\Psi^i) \cdot \sigma) - \nabla(\Psi^i) \cdot \sigma \cdot C\\
&=&\left(-\frac{1}{2} \tau' \sigma'^i- \sigma'^i \cdot C' \right) \circ \Psi.
\end{eqnarray*}
Since $\Psi$ is surjective, the thesis follows.\\
${}\hfill$ \end{proof}

\begin{remark}
If $V_1=(Y_1,0,0),...,V_k=(Y_k,0,0)$ are strong symmetries, conditions \eqref{pippo} of Theorem \ref{reduction_theorem2} on $V=(Y,C,\tau)$ can be rewritten as
\begin{equation}\label{equation_commutation}
[V,V_i]=\sum_{j=1}^k \lambda^i_j(x) V_j,
\end{equation}
where  $\lambda_i^j(x) $ are  smooth functions (in the particular case of $V_1,...,V_k,V$ generating a finite dimensional Lie algebra, $\lambda^i_j$ are  constants).
In general, if $V_i$ are quasi-strong symmetries satisfying the hypotheses of Theorem \ref{theorem_reduction1}, it is possible to prove an analogous
of Theorem \ref{reduction_theorem2} using only hypotheses \eqref{equation_commutation} but $C'$ in the reduced symmetry $(Y',C',\tau')$ satisfies $C'\circ \Psi=C+\tilde C$ where $\tilde{C}$ is such that
 $Y_i(C+\tilde{C})=0$ for $i=1,...,k$ and $\nabla(\Psi) \cdot \sigma \cdot \tilde{C}=0$.
\end{remark}

\subsection{Reconstruction }\label{reduction_subsection}

In this section we discuss the problem of reconstructing a process starting from the knowledge of the  reduced one.
In order to do this we need the following definition mainly  inspired by ODEs framework.

\begin{definition}\label{definition_reconstruction}
Let $X$ and $Z$ be two  processes on $M$ and $M'$ respectively. We say that $X$ can be reconstructed from $Z$ until the stopping time $\tau$
if there exists a smooth function $F:\mathbb{R}^{k(m+1)} \times M' \times M \rightarrow M$ such that
$$X^{\tau}_t=F\left(\int_0^{t \wedge \tau}{f_0(s,Z_s)ds},\int_0^{t \wedge \tau}{f_1(s,Z_s)dW^1_s},...,\int_0^{t \wedge \tau}{f_m(s,Z_s)dW^m_s},\
Z_{t \wedge \tau}  , \ X_0\right)$$
where    $W^1,...,W^m$ are Brownian motions and   $ f_i:\mathbb{R} \times M' \rightarrow \mathbb{R}^k$ are smooth functions.
The process $X$ can be progressively reconstructed from $Z$ until the stopping time $\tau$ if there are some real processes $Z^1,...,Z^r$, such that every $Z^i$ can be reconstructed until the stopping time $\tau$ starting from the process $(Z^1,...,Z^{i-1},Z)$, and $X$ can be reconstructed until the stopping time $\tau$ from the process $(Z^1,...,Z^r,Z)$.
\end{definition}
We remark  that Definition \ref{definition_reconstruction} is general enough for our purposes (as we only consider Brownian motion driven SDEs), but can be easily generalized
to include integration with respect more general stochastic processes.

\begin{theorem}\label{theorem_solvable}
Let $\{Y_1,...,Y_r\}$ be a set of regular reduction vector fields for the reduction map $\Psi:M \to M'$ such that $Y_1,...,Y_r$ generate a solvable algebra of  strong symmetries for the SDE $(\mu,\sigma)$
 and let $X^x$ be  the unique solution to the
 SDE $(\mu,\sigma)$ with Brownian motion $W$ such that $X^x_0=x$ almost surely. Then,
 for any $x \in M$, there exists a stopping time $\tau_x$ almost surely positive such that the process $X^x$ can be progressively reconstructed from $\Psi(X^x)$.
\end{theorem}
\begin{proof}
Let $\Phi$ be the diffeomorphism given by Theorem \ref{theorem_solvable_coordinate}, defined in a neighborhood $U(x_0)$, and $T=(\Phi, I_m, 1)$.
If  $(\tilde{\mu},\tilde{\sigma})=E_T(\mu,\sigma)$ then   $\Phi_*(Y_1),...,\Phi_*(Y_r)$ are symmetries of $(\tilde{\mu},\tilde{\sigma})$
 in canonical form and, denoting by $\tilde{x}^i$ the coordinate system on $\tilde{M}=\Phi(U)$, we have
$$\frac{\partial}{\partial{\tilde x}_i}(\tilde{\mu}^j)=\frac{\partial}{\partial{\tilde x}_i}(\tilde{\sigma}^j_{\alpha})=0,$$
for $i \leq r$ and  $j \leq i$. This means that the $r$-th row of the SDE $(\tilde{\mu},\tilde{\sigma})$ does not depend on $\tilde{x}^1,...,\tilde{x}^r$, the $(r-1)$-th
row does not depend on $\tilde{x}^1,...,\tilde{x}^{r-1}$ and so on. Hence the process $\tilde{X}=\Phi(X^{x_0})$ can be progressively reconstructed
from $\Pi(\tilde{X})$, where $\Pi$ is the projection of $\tilde{M}$ on the last $n-r$ coordinates. Since by definition of $\tilde X$ and $
\Phi$ we have  $\Pi(\tilde{X})=\Pi(\Phi(X^{x_0}))=\Psi(X^{x_0})$, the process $\tilde{X}$ can be progressively reconstructed from $\Psi(X^{x_0})$.
Moreover, being $X^{x_0}=\Phi^{-1}(\tilde{X})$ until the process $X^{x_0}$ exits from the open set $U$ we have that $\tilde{X}^{x_0}$ can be progressively reconstructed from $\Psi(X)$ until the stopping time
$$\tau=\inf_{t \in \mathbb{R}_+}\{X^{x_0}_t \not \in U\},$$
that is almost surely positive since $U$ is a neighborhood of $x_0$.
${}\hfill$ \end{proof}

\begin{corollary}
 In the hypotheses  and with the notations of Theorem \ref{theorem_solvable} if the Lie algebra generated by $Y_1,...,Y_r$ is Abelian then $X^{x}$ can be reconstructed from $\Psi(X^{x})$.
\end{corollary}
\begin{proof}
If $\mathcal{G}=\spann\{Y_1,...,Y_r\}$ is Abelian, the diffeomorphism $\Phi$ of Theorem \ref{theorem_solvable_coordinate} rectifies $\mathcal{G}$.
${}\hfill$ \end{proof}

\begin{remark}
In order to compare our results with the reconstruction method
 proposed in \cite{Lazaro_Ortega} we consider the case of $Y_1,...,Y_r$ generating a general  Lie group $G$ whose action is free and proper.
 In this case $(M,\Psi,M')$ is a principal bundle with structure group $G$ and locally diffeomorphic to $U=V \times G$, where $V$ is an open subset of $\mathbb{R}^{n-r}$.
 If we denote by $\bar{X}$ the reduced process and by $g$ the coordinates on $G$ we have that the process $\bar{G}$ in $U$ satisfies the following Stratonovich equation
\begin{equation}\label{equation_reconstruction}
d\bar{G}_t=\sum_{i=1}^r f_0^i(\bar{X}_t) Y_i(\bar G_t) dt+\sum_{\alpha=1}^m \sum_{i=1}^r f_{\alpha}^i(\bar{X}_t) Y_i(\bar G_t) \circ dW^{\alpha}_t,
\end{equation}
where $f^i_j$ are smooth real-valued functions. Despite the fact that  the knowledge of the reduced process $\bar{G}_t$ formally allows the  reconstruction of the initial  process $X_t$, for
 a general group $G$ this reconstruction cannot be reduced to quadratures.\\
 On the other hand, if
 $G$ is  solvable, it is possible to choose a set of global coordinates on $G$ which reduce equation \refeqn{equation_reconstruction} to
 integration by quadratures as required by Definition \ref{definition_reconstruction}.\\
\end{remark}
The following definition  generalizes to the stochastic framework the well known
definition of integrability for a system of ODEs.
\begin{definition}
A SDE $(\mu,\sigma)$ is  completely integrable (or simply integrable) if for any  $x \in M$ there exists an  almost surely positive stopping time $\tau_x>0$ such that the solution process $X^x$ can be progressively reconstructed until the stopping time $\tau_x$ from a deterministic process.\\
\end{definition}

\begin{theorem}\label{theorem_integrability}
Let $(\mu,\sigma)$ be a SDE on $M\subset \mathbb{R}^n$ admitting an $n$-dimensional solvable Lie algebra $\mathcal{G}$ of strong symmetries which are also a regular set of vector fields. Then  $(\mu,\sigma)$ is integrable.
\end{theorem}
\begin{proof}
 Since  $\mathcal{G}$  has the same dimension of $M$, the map $\Psi=0$, and the transformed SDE  $(\mu',\sigma')$
is such that  the first row of $\mu'$ and $\sigma'$ does not depend on $x'^1$, the second row does not depend on $x'^1,x'^2$ and so on. Therefore the $n$-th
row of $\mu'$ and $\sigma'$ does not depend on any variables $x'^1,...,x'^n$ and so it is constant. This means that the solution $X'=\Phi(X)$ can be progressively reconstructed from a constant process.
${}\hfill$ \end{proof}\\

\noindent We remark  that the hypotheses of Theorem \ref{theorem_integrability} are only sufficient and not necessary. For example the SDE
\begin{eqnarray*}
\left(\begin{array}{c}
dX_t\\
dZ_t\end{array} \right)&=& \left(\begin{array}{c}
0\\
g(X_t) \end{array} \right)+\left(\begin{array}{c}
1\\
f(X_t)\end{array} \right)dW_t,
\end{eqnarray*}
is obviously integrable, but it is easy to prove that (for general $g(x),f(x)$) it does not admit other symmetries than
$$V=\left(\left(\begin{array}{c}
0\\
1
\end{array}
\right), \left(\begin{array}{cc}
0 & 0 \\
0 & 0 \end{array}\right), 0\right).$$

\section{ Examples }\label{section_examples}

In this section we apply our general reduction procedure to some  explicit examples.
Following the line of previous discussion, given a SDE $(\mu, \sigma)$, we start by looking for   a solvable algebra of symmetries
$\mathcal{G}= \{ V_1, \ldots, V_r \}$ for $(\mu, \sigma)$. Hence we compute   a stochastic transformation $T=(\Phi, B,  \eta)$
transforming $ V_1, \ldots, V_r $ into strong symmetries $V'_k=(Y'_k, 0,0)$ for the transformed SDE $E_T(\mu, \sigma)$ such that the vector fields $Y'_1, \ldots , Y'_r$ are in canonical form. Finally we   use the results of Section
\ref{section_reduction_reconstruction} to reduce (or integrate) the transformed SDE $E_T(\mu, \sigma)$  and we  reconstruct the solution to
$(\mu, \sigma)$ by means of the inverse transformation $T^{-1}$.

\subsection{A class of one-dimensional Kolmogorov-Pearson diffusions}\label{subsection_linear}

We consider the following class of SDE within the Kolmogorov-Pearson type diffusions
\begin{equation}\label{equation_linear}
dX_t=(\lambda X_t + \nu) dt+ \sqrt{\alpha X^2_t+2\beta X_t+\gamma}  \ dW_t,
\end{equation}
where $\alpha,\beta,\gamma,\lambda,\nu \in \mathbb{R}$, $\alpha\ge 0$ and $\alpha \gamma -\beta^2 \geq 0$.  \\
For $\alpha =\beta =0$ the class  includes the Ornstein-Uhlenbeck process and for $\alpha \gamma -\beta ^2=0$ the important class of  one-dimensional general linear SDEs of the form
\begin{equation}\label{standardlinear}
dX_t=(\lambda X_t+\nu)+\left(\sqrt{\alpha}X_t+\frac{\beta}{\sqrt{\alpha}}\right) dW_t.
\end{equation}
 Beyond the large  number of applications of Ornstein-Uhlenbeck process and of linear SDEs and
their spatial transformations, the  Kolmogorov-Pearson class \eqref{equation_linear} has notable applications to finance (see \cite{Borland2002,Shaw2015}),
physics  (see \cite{Friedrich1997,Wilk2000}) and  biology (see \cite{Ghasemi2006}). Moreover,  there is a growing interest in the study of statistical inference
(see \cite{Forman2008}), in the analytical and spectral properties of the  Kolmogorov equation  associated with \eqref{equation_linear} (see \cite{Avram1}) and in the
development of efficient numerical algorithms for its numerical simulation (see \cite{Bormetti2014}). Finally the Kolmogorov-Pearson
diffusions are  examples of \virgolette{polynomial processes} that are becoming quite  popular in financial mathematics (\cite{polynomial_processes}).
For many particular values of the parameters $\alpha,\beta,\gamma,\lambda,\nu$ it is well known that equation \refeqn{equation_linear}  is an integrable SDE
(first of all in the
standard linear case corresponding to $\alpha \gamma-\beta^2=0$).
Anyway this integrability property cannot be directly related to the existence of strong symmetries as showed by  the following proposition.

\begin{proposition}\label{proposition_linear1}
The SDE $(\lambda x+\nu,\sqrt{\alpha x^2+2\beta x+\gamma})$ admits strong symmetries if and only if
\begin{eqnarray*}
2\beta \nu-2\gamma \lambda+\alpha \gamma-\beta^2&=&0\\
\alpha \nu-\beta \lambda&=&0.
\end{eqnarray*}
\end{proposition}
\begin{proof}
 The  determining equations \eqref{determining_eq} for a strong symmetry $V=(Y,0,0)$  of \refeqn{equation_linear}, with $Y=Y^1\partial_x$, are
\begin{eqnarray}
&\frac{(\alpha x+\beta)}{\sqrt{\alpha x^2+2\beta x+\gamma}}Y^1-\partial_x(Y^1)\sqrt{\alpha x^2+2\beta x+\gamma}=0&\label{equation_linear7}\\
&\lambda Y^1-\frac{(\alpha x^2+2\beta x+\gamma)}{2}\partial_{xx}(Y^1)-\partial_x(Y^1)(\lambda x+\nu)=0.&\label{equation_linear8}
\end{eqnarray}
Equation \refeqn{equation_linear7} is an ODE in $Y^1$ with solution
\begin{equation}\label{equation_linear9}
Y^1=Y^1_0\sqrt{\alpha x^2+2\beta x+\gamma},
\end{equation}
where $Y^1_0 \in \mathbb{R}$. Inserting the expression \refeqn{equation_linear9} in \refeqn{equation_linear8} we obtain
$$Y^1_0\frac{(2\beta \nu-2\gamma \lambda+\alpha \gamma -\beta^2)+2x(\alpha \nu-\beta \lambda)}{2\sqrt{\alpha x^2+2\beta x+\gamma}}=0$$
and this concludes the proof.
${}\hfill$\end{proof} \\
We remark that a standard linear SDE of the form \eqref{standardlinear} admits a symmetry if and only if
$\alpha \nu-\beta \lambda=0$. Therefore, in spite of their  integrability,  standard linear SDEs do not have,  in general, strong symmetries.\\
In order to apply  a symmetry approach  to the study of the integrability of  \refeqn{equation_linear}, we consider the  following two-dimensional system:
\begin{equation} \label{equation_linear2}
\left(\begin{array}{c}
dX_t\\
dZ_t\end{array}\right)=
\left(\begin{array}{c}
\lambda X_t+\nu\\
\lambda Z_t\end{array} \right)dt+\left(
\begin{array}{cc}
\sqrt{\alpha  X_t^2+2\beta  X_t+ \gamma} & 0\\
\frac{Z_t(\alpha X_t+\beta )}{\sqrt{\alpha X_t^2+2\beta X_t+\gamma}} & \frac{Z_t\sqrt{\alpha \gamma-\beta^2}}{\sqrt{\alpha X_t^2+2\beta X_t+\gamma}} \end{array}\right)
\left( \begin{array}{c}
dW^1_t \\
dW^2_t \end{array} \right),
\end{equation}
where $W^1_t:=W_t$. In the standard linear case,  system \eqref{equation_linear2} consists of  SDE \refeqn{standardlinear} and of the associated homogeneous one.
If we look for the symmetries  of  system (\ref{equation_linear2}) of the form  $V=(Y,C,\tau)$, where $Y=(Y^1,Y^2)$, $C=\left(\begin{array}{cc}
0 & c(x,z)\\
-c(x,z) & 0 \end{array}\right)$ and $\tau(x,z)$,  the determining equations are:
$$
\begin{array}{c}
\frac{(\alpha x+\beta)}{\sqrt{\alpha x^2+2\beta x+\gamma}}Y^1-\partial_x(Y^1)\sqrt{\alpha x^2+2\beta x+\gamma}-\partial_z(Y^1)\frac{z(\alpha x+\beta )}
{\sqrt{\alpha x^2+2\beta x+\gamma}}+\\
+\frac{\tau}{2} \sqrt{\alpha x^2+2\beta x+\gamma}=0
\end{array}
$$

$$
\begin{array}{c}
\frac{z(\alpha \gamma -\beta^2)}{(\sqrt{\alpha x^2+2\beta x+\gamma})^3}Y^1+\frac{\alpha x+\beta }{\sqrt{\alpha x^2+2\beta x+\gamma }}
Y^2-\partial_x(Y^2)\sqrt{\alpha x^2+2\beta x+\gamma }+\\
-\partial_z(Y^2)\frac{z(\alpha x+\beta )}{\sqrt{\alpha x^2+2\beta x+\gamma}}+\frac{\tau}{2} \frac {z(\alpha x+\beta)}{\sqrt{\alpha x^2+2\beta x+\gamma}} -  \frac{cz\sqrt{\alpha\gamma-\beta^2}}{\sqrt{\alpha x^2+2\beta x+\gamma}}=0
\end{array}
$$

$$
\begin{array}{c}
-\frac{z\sqrt{\alpha \gamma-\beta^2}}{\sqrt{\alpha x^2+2\beta x+\gamma}}\partial_z(Y^1)+c \sqrt{\alpha x^2+2\beta x+\gamma} =0\\
\end{array}
$$
$$
\begin{array}{c}
-\frac{z\sqrt{\alpha \gamma -\beta^2}(\alpha x+\beta)}{(\sqrt{\alpha x^2+2\beta x+\gamma})^3}Y^1+\frac{\sqrt{\alpha \gamma-\beta^2}}{\sqrt{\alpha x^2+2\beta x+\gamma}}Y^2-\frac{z\sqrt{\alpha \gamma-\beta^2}}{\sqrt{\alpha x^2+2\beta x+\gamma}}\partial_z(Y^2)\\
+c \frac{z(\alpha x+\beta )}{\sqrt{\alpha x^2+2\beta x+\gamma}}+\frac {\tau}{2}\frac {z\sqrt{\alpha \gamma-\beta^2}}{\sqrt{\alpha x^2+2\beta x+\gamma}}=0
\end{array}
$$
$$
\begin{array}{c}
\lambda Y^1-(\lambda x+\nu)\partial_x(Y^1)-\lambda z \partial_z(Y^1)-\frac{\alpha x^2+2\beta x+\gamma}{2}\partial_{xx}(Y^1)-\frac{\alpha z^2}{2}\partial_{zz}(Y^1)+\\
-z(\alpha x+\beta)\partial_{xz}(Y^1) + \tau (\lambda x+\nu)=0
\end{array}
$$
$$
\begin{array}{c}
\lambda Y^2-(\lambda x+\nu)\partial_x(Y^2)-\lambda z \partial_z(Y^2)-\frac{\alpha x^2+2\beta x+\gamma}{2}\partial_{xx}(Y^2)-\frac{\alpha z^2}{2}\partial_{zz}(Y^2)+\\
-z(\alpha x+\beta )\partial_{xz}(Y^2) + \tau \lambda z=0.
\end{array}
$$
We can easily solve the previous overdetermined system of PDEs by a computer algebra software and we find
 two quasi-strong symmetries
\begin{eqnarray*}
V_1=(Y_1,C_1,\tau_1)&=&\left(\left(
\begin{array}{c}
z \\
0
\end{array}\right),
\left(\begin{array}{cc}
0 &\frac{z\sqrt{\alpha \gamma -\beta^2}}{\alpha x^2+2\beta x+\gamma}\\
-\frac{z\sqrt{\alpha \gamma -\beta^2}}{\alpha x^2+2\beta x+\gamma} & 0
\end{array}\right),0 \right)\\
V_2=(Y_2,C_2,\tau_2)&=&\left(\left(
\begin{array}{c}
0 \\
z
\end{array}\right),
\left(\begin{array}{cc}
0 & 0\\
0 & 0
\end{array}\right),0 \right).\\
\end{eqnarray*}
Therefore the function  $\tilde{\Psi}:M \rightarrow \mathbb{R}$ given by  $\tilde{\Psi}(x,z)=x$ is a reduction function with respect to the strong
symmetry $Y_2$,  being  $\nabla(\tilde{\Psi}) \cdot Y_2=Y_2(\tilde{\Psi})=0$,  and
the reduced equation on $M'=\tilde{\Psi}(M)=\mathbb{R}$ is exactly the original SDE \refeqn{equation_linear}. \\
This circumstance partially explains Proposition \ref{proposition_linear1}: since the original SDE \refeqn{equation_linear} turns out as the reduction
of the integrable system
\refeqn{equation_linear2} with respect to the \virgolette{wrong} symmetry, it does not inherit any symmetry. \\
In order to integrate system  \refeqn{equation_linear2} and therefore also the original equation \refeqn{equation_linear}, we start by looking for a
stochastic transformation $T=(\Phi, B, \eta)$ such that $T_*(V_1)$ and $T_*(V_2)$ are strong transformations and $\Phi_*(Y_1)$ and $\Phi_*(Y_2)$ are in
canonical form. Since $V_1,V_2$ are quasi-strong infinitesimal stochastic transformations we can restrict to  a quasi-strong transformation  $T=(\Phi, B, 1)$.  \\
Following  the explicit construction of  Theorem \ref{theorem_solvable_coordinate} the function  $\Phi$ turns out to be both globally defined  and globally
invertible on $M$.\\
The vector fields $Y_1,Y_2$, whose flows are defined by
\begin{eqnarray*}
\Phi^1_{a^1}(x,z)&=&\left( \begin{array}{c}
x+a^1z\\
z \end{array}\right)\\
\Phi^2_{a^2}(x,z)&=&\left(\begin{array}{c}
x\\
e^{a^2} z\end{array} \right),
\end{eqnarray*}
generate a free and proper action of a solvable simply connected non Abelian Lie group on $M$.
Hence, if we consider the point $p=(0,1)^T$ and  the function   $F:\mathbb{R}^2 \to M$ given by
$$F(a^1,a^2)=\Phi^1_{a^1}(\Phi^2_{a^2}(p))=\left(\begin{array}{c}
a^1 e^{a^2} \\
e^{a^2}
\end{array} \right),$$
the function  $\Phi:M \rightarrow \mathbb{R}^2$, which  is the inverse of $F$,  is given by
$$\Phi(x,z)=\left(\begin{array}{c}
\frac{x}{z}\\
\log(z)
\end{array} \right).$$
By Theorem \ref{theorem_infinitesimal_SDE1} the equations for $B$ are
\begin{eqnarray*}
z\partial_x(B )&=&-B \cdot C_1\\
z\partial_z(B)&=&0.
\end{eqnarray*}
Writing:
\begin{equation}\label{matrixB}
B=\left(\begin{array}{cc}
b & \sqrt{1-b^2}\\
- \sqrt{1-b^2} & b
\end{array}\right)
\end{equation}
from the second equation we deduce that $B$ does not depend on $z$ and from the first one we obtain that $b$ satisfies the equation
$$\partial_x b =\frac{\sqrt{\alpha \gamma -\beta^2}}{\alpha x^2+2\beta x+\gamma} \sqrt{1-b^2}.$$
The latter equation is an ODE with separable variables admitting the following solution
$$b=\frac{\alpha x+\beta }{\sqrt{\alpha }\sqrt{\alpha x^2+2\beta x+\gamma}},$$
and we get
$$B=\left(\begin{array}{cc}
\frac{\alpha x+\beta }{\sqrt{\alpha }\sqrt{\alpha x^2+2\beta x+\gamma}} & \frac{\sqrt{\alpha \gamma-\beta^2}}{\sqrt{\alpha }\sqrt{\alpha x^2+2\beta x+\gamma}}\\
-\frac{\sqrt{\alpha \gamma-\beta^2}}{\sqrt{\alpha}\sqrt{\alpha x^2+2\beta x+\gamma}} & \frac{\alpha x+\beta }{\sqrt{\alpha }\sqrt{\alpha x^2+2\beta x+\gamma}}
\end{array}\right).$$
Putting $(x',z')^T=\Phi(x,z)$ we have
\begin{eqnarray*}
V'_1=T_*(V_1)&=&\left(\Phi_*(Y_1)=\left( \begin{array}{c}
1\\
0
\end{array}\right), \left( \begin{array}{cc}
0 & 0 \\
0 & 0 \end{array} \right), 0 \right),\\
V'_2=T_*(V_2)&=&\left(\Phi_*(Y_2)=\left(\begin{array}{c}
-x'\\
1
\end{array}\right), \left( \begin{array}{cc}
0 & 0 \\
0 & 0 \end{array}\right), 0 \right),
\end{eqnarray*}
and so $Y'_1,Y'_2$ are two generators of $\Phi_*(\mathcal{G})$ in  canonical form.
Introducing  $dW'_t = B(X_t,Z_t) \cdot dW_t$ and applying Ito formula respectively to $\phi_1=x/z$ and $\phi_2=\log(z)$ we  can write
 equation \eqref{equation_linear2} in the new variables:
\begin{eqnarray*}
dX'_t&=&(\nu-\beta )e^{-Z'_t}dt+\frac{\beta e^{-Z'_t}}{\sqrt{\alpha }}dW'^1_t-\sqrt{\frac{\alpha \gamma -\beta^2}{\alpha}}e^{-Z'_t}dW'^2_t\\
dZ'_t&=&\left(\lambda-\frac{\alpha}{2}\right)dt+\sqrt{\alpha}dW'^1_t.
\end{eqnarray*}
Since this is  an integrable SDE, the solutions to equation \refeqn{equation_linear2}
can be recovered following our general procedure.
In particular, when $\alpha \gamma -\beta^2=0$, the solution to equation \refeqn{equation_linear} is given by
\begin{eqnarray*}
Z_t&=&Z_0e^{(\lambda-\alpha /2)t+\sqrt{\alpha }W_t}\\
X_t&=&Z_t\left(\frac{X_0}{Z_0}+\int_0^t{\frac{\nu-\beta }{Z_s}ds}+\int_0^t{\frac{\beta}{\sqrt{\alpha}Z_s}dW_s}\right)
\end{eqnarray*}
which is the well known explicit solution to the (general) linear one-dimensional SDE  \refeqn{standardlinear}.

\subsection{Integrability of  a singular SDE}\label{subsection_singular}

Let us consider the SDE on $M=\mathbb{R}^2 \backslash \{(0,0)^T\}$
\begin{equation}\label{equation_singular}
\left(\begin{array}{c}
dX_t\\
dZ_t
\end{array}\right)=\left(
\begin{array}{c}
\frac{\alpha X_t}{X_t^2+Z_t^2}\\
\frac{-\alpha Z_t}{X_t^2+Z_t^2}
\end{array} \right)dt +
\left(\begin{array}{cc}
\frac{X_t^2-Z_t^2}{\sqrt{X_t^2+Z_t^2}} & 0\\
0 & \frac{X_t^2-Z_t^2}{\sqrt{X_t^2+Z_t^2}}
\end{array}\right) \cdot \left(\begin{array}{c}
dW^1_t\\
dW^2_t
\end{array}\right),
\end{equation}
where $\alpha \in \mathbb{R}$. Despite the coefficients of $(\mu, \sigma)$ have a singularity in $(0,0)^T$, we will prove that  the solution to \eqref{equation_singular}
is not singular and that the explosion time of \eqref{equation_singular}  is $+\infty$
for any deterministic initial condition $X_0 \in M$. \\
A symmetry $V=(Y,C,\tau) $ of \eqref{equation_singular}, with $Y=(Y^1,Y^2)$, $C=\left(\begin{array}{cc}
0 & c(x,z)\\
-c(x,z) & 0 \end{array}\right)$ and $\tau(x,z)$, has to satisfy the following determining equations

$$
\begin{array}{c}
\frac{x^3+3xz^2}{(x^2+z^2)^{3/2}}Y^1-\frac{z^3+3zx^2}{(x^2+z^2)^{3/2}}Y^2-\frac{x^2-z^2}{\sqrt{x^2+z^2}}\partial_x(Y^1)+\frac{1}{2}\tau \frac{x^2-z^2}{\sqrt{x^2+z^2}}=0
\end{array}$$

$$
\begin{array}{c}
-\frac{x^2-z^2}{\sqrt{x^2+z^2}}\partial_x(Y^2)-c\frac{x^2-z^2}{\sqrt{x^2+z^2}}=0
\end{array}$$

$$
\begin{array}{c}
-\frac{x^2-z^2}{\sqrt{x^2+z^2}}\partial_z(Y^1)+c\frac{x^2-z^2}{\sqrt{x^2+z^2}}=0
\end{array}$$

$$
\begin{array}{c}
\frac{x^3+3xz^2}{(x^2+z^2)^{3/2}}Y^1-\frac{z^3+3zx^2}{(x^2+z^2)^{3/2}}Y^2-\frac{x^2-z^2}{\sqrt{x^2+z^2}}\partial_z(Y^2)+\frac{1}{2}\tau\frac{x^2-z^2}{\sqrt{x^2+z^2}}=0
\end{array}$$

$$
\begin{array}{c}
-\frac{\alpha(x^2-z^2)}{(x^2+z^2)^2}Y^1-\frac{2 \alpha xz}{(x^2+z^2)^2}Y^2-\frac{\alpha x}{x^2+z^2}\partial_x(Y^1)
\\
+\frac{\alpha z}{x^2+z^2}\partial_z(Y^1)-\frac{1}{2}\frac{(x^2-z^2)^2}{x^2+z^2}(\partial_{xx}(Y^1)+\partial_{zz}(Y^1))+\tau \frac{\alpha x}{x^2+z^2}=0
\end{array}$$

$$
\begin{array}{c}
\frac{2\alpha xz}{(x^2+z^2)^2} Y^1-\frac{\alpha(x^2-z^2)}{(x^2+z^2)^2}Y^2-\frac{\alpha x}{x^2+z^2}\partial_x(Y^2)\\
+\frac{\alpha z}{x^2+z^2}\partial_z(Y^2)-\frac{1}{2}\frac{(x^2-z^2)^2}{x^2+z^2}(\partial_{xx}(Y^2)+\partial_{zz}(Y^2))-\tau \frac{\alpha z}{x^2+z^2}=0
\end{array}$$

Solving this system of PDEs by a computer algebra software we find the
unique (quasi-strong) symmetry
$$
V=(Y,C,\tau)=\left(\left(\begin{array}{c}
\frac{z}{x^2+z^2}\\
\frac{x}{x^2+z^2}
\end{array} \right), \left(\begin{array}{cc}
0 & \frac{x^2-z^2}{(x^2+z^2)^2}\\
-\frac{x^2-z^2}{(x^2+z^2)^2} & 0
\end{array}\right),0\right)
$$
which unfortunately does not generate a one parameter group of stochastic transformations,
as the trajectories of the points of the form $(h,h)$ and $(h,-h)$ (with  $h \in \mathbb{R} \backslash \{0\})$, reach $(0,0)$ in a finite time.\\
Hence, in order to find a stochastic transformation $T=(\Phi,B,\eta)$ such that $T_*(V)$ is a strong transformation and $\Phi_*(Y)$ is in canonical form  we have to solve the equations $Y(\Phi)=(1,0)^T$ and  $Y(B)=-B \cdot C$. Writing $\Phi=(\tilde{\Phi},\Psi)^T$, the map $\Psi$ has to solve
$Y(\Psi)=0$, i.e.
$$\frac{z}{x^2+z^2}\partial_x(\Psi)+\frac{x}{x^2+z^2}\partial_z(\Psi)=0.$$
By using the method of the characteristics,  we immediately obtain a particular solution
$$\Psi(x,z)=x^2-z^2.$$
In order to find an adapted coordinate system for the Abelian Lie algebra $\mathcal{G}=\spann\{Y\}$ we have to solve the equation
$Y(\tilde{\Phi})=1$, i.e.
$$z \partial_x(\tilde{\Phi})+x\partial_z(\tilde{\Phi})=x^2+z^2.$$
Once again,  applying the method of the characteristics,  we obtain
$$\tilde{\Phi}(x,z)=xz, $$
so that we can consider the local diffeomorphism of $M=\mathbb{R}^2 \backslash \{(0,0)^T\}$
$$\Phi(x,z)=\left( \begin{array}{c}
\tilde{\Phi}\\
\Psi \end{array} \right)=\left( \begin{array}{c}
xz\\
x^2-z^2
\end{array}\right)$$
which is only locally invertible. \\
To construct the matrix-valued function $B$ of the form \eqref{matrixB} we solve the equation $Y(B)=-B \cdot C$. In the new coordinates $(x',z')=\Phi(x,z)$ the equation   becomes
$$\partial_{x'}b=\frac{z'}{4x'^2+z'^2} \sqrt{1-b^2},$$
whose solution is
$$b=\pm\frac{\sqrt{\sqrt{z'^2+4x'^2}-z'}}{\sqrt{2}(z'^2+4x'^2)^{1/4}}$$
and, coming back to the original coordinate system,  we find
$$B=\left(\begin{array}{cc}
\frac{z}{\sqrt{x^2+z^2}} & \frac{x}{\sqrt{x^2+z^2}}\\
\frac{-x}{\sqrt{x^2+z^2}} & \frac{z}{\sqrt{x^2+z^2}}\end{array}\right).$$
The transformed SDE $(\mu',\sigma')=E_T(\mu,\sigma)$ has coefficients
\begin{eqnarray*}
\mu'&=&L(\Phi) \circ \Phi^{-1}=\left( \begin{array}{c}
0\\
2\alpha \end{array}\right)\\
\sigma'&=&(\nabla(\Phi) \cdot \sigma) \circ \Phi^{-1}=\left( \begin{array}{cc}
z' & 0\\
0 & -2z' \end{array}\right)
\end{eqnarray*}
and, by applying Ito formula, the original two-dimensional SDE becomes
\begin{equation}\label{equation_singular7}
\begin{array}{ccl}
dX'_t&=&Z'_t dW'^1_t\\
dZ'_t&=&2\alpha dt- 2 Z'_t dW'^2_t,
\end{array}
\end{equation}
where $dW'_t=B(X_t,Z_t) \cdot dW_t$. Since the equation in $Z'$ is linear, the above SDE is integrable and therefore also \refeqn{equation_singular}
is integrable. Furthermore, since the map $\Phi:M \rightarrow \mathbb{R}^2 \backslash \{(0,0)^T\}$ is a double covering map and since the SDE
\refeqn{equation_singular7} has explosion time $\tau=+\infty$, the SDE \refeqn{equation_singular},  although singular at the origin,  has also
explosion time $\tau=+\infty$  for any deterministic initial condition $X_0 \in M$. \\
This example point out  the importance of developing a local reduction theory for  SDEs, since in this case a
global approach cannot be successful.

\subsection{Stochastic perturbation of mechanical equations}\label{subsection_mechanics}

In this example we analyze a wide  class of models, related to  (stochastic) mechanics, of the form
\begin{equation}\label{equation_mechanics1}
\begin{array}{ccl}
dX^i_t&=& V^i_t dt\\
dV^i_t&=&F^i_0(X_t,V_t)dt+\sum_{\alpha} F^i_{\alpha}(X_t) dW^{\alpha}_t,
\end{array}
\end{equation}
i.e. with SDE coefficients:
\begin{eqnarray*}
\mu&=&\left(\begin{array}{c}
v^1\\
\vdots \\
v^n\\
F^1_0(x,v)\\
\vdots\\
F^n_0(x,v)\end{array}\right)\\
\sigma&=&\left(\begin{array}{ccc}
0 & \cdots & 0\\
\vdots & \vdots & \vdots\\
F^1_1(x) & \cdots & F^1_m(x)\\
\vdots & \vdots & \vdots \\
F^n_1(x) & \cdots & F^n_m(x) \end{array}\right),
\end{eqnarray*}
where $(x^i,v^i)$ is the standard coordinate system of $M=\tilde{M} \times \mathbb{R}^n$ and  $\tilde{M}$ is an open set of $\mathbb{R}^n$.\\
This kind of SDEs, representing   a stochastic perturbation of the Newton equations for $n$ particles of mass $m_i=1$  subjected to forces
depending on the positions and on the velocities
\begin{equation}\label{equation_mechanics2}
\frac{d^2X^i_t}{dt^2}=F^i\left(X_t,\frac{dX_t}{dt}\right),
\end{equation}
arise in many contexts of mathematical physics. The class includes the Langevin type equation  
often used in the
 framework of Stochastic Thermodynamics (see, e.g.,
\cite{Marconi2008,Seifert2012}) for $F^i_{\alpha}=\delta^i_{\alpha}$ and
 $F^i_0=-\gamma V^i+ \partial_i(U)(x)$, where  $U:\mathbb{R}^n \rightarrow \mathbb{R}$ is a smooth function and
$\gamma \in \mathbb{R}_+$.
Furthermore, if the forces $F^i_0$ arise from a Lagrangian $L$ of the form
$$L=\frac{1}{2}\sum_{i,j}g_{i,j}(x)v^iv^j-U(x)$$
(where $g_{i,j}(x)$ is a metric tensor on $\mathbb{R}^n$) and the random perturbations $F^i_{\alpha}$ are given by
$$F^i_{\alpha}=\sum_jg^{i,j}(x)\partial_j(U_{\alpha}),$$
with $U_{\alpha}$ smooth functions,  \refeqn{equation_mechanics1} turns out to be a Lagrangian system with  the following action functional
$$S=\int_0^t{L(X_s,V_s)ds}+\sum_{\alpha}\int_0^t{U_{\alpha}(X_s)dW^{\alpha}_s}$$
(see, e.g.,\cite{Bismut1981,Ortega2008}). There is a growing interest for this kind of stochastic perturbations
of Lagrangian and Hamiltonian systems due both to  their special mathematical properties and to their applications in mathematical physics
(see, e.g., \cite{Arnaudon_Holm2016,Cruzeiro2014,Ratiu2015,Holm2015,Lescot,Li_integrable,Zambrini2010}).\\
In the following we propose  a method to obtain a SDE of the form \refeqn{equation_mechanics1} which can be interpreted as a symmetric stochastic perturbation of a symmetric  ODE
of the form \refeqn{equation_mechanics2}.\\
Given  a vector field $\tilde{Y}_0=(\tilde{Y}^1_0(x),...,\tilde{Y}^n_0(x))^T$ on $\tilde{M}$ which  is a symmetry of \refeqn{equation_mechanics2}, the vector field
$$Y=\left(\begin{array}{c}
\tilde{Y}_0^1(x) \\
\vdots \\
\tilde{Y}_0^n(x) \\
\sum_{k=1}^n \partial_{x^k}(\tilde{Y}_0^1) v^k\\
\vdots\\
\sum_{k=1}^n \partial_{x^k}(\tilde{Y}_0^n) v^k \end{array}\right),$$
is a symmetry of  \refeqn{equation_mechanics1}, when $F^i_{\alpha}=0$. \\
If $\tilde{Y}_{\alpha}=(\tilde{Y}^1_{\alpha}(x),...,\tilde{Y}^n_{\alpha}(x))$, for $\alpha=1,...,m$, are
 $m$ vector fields on $\tilde{M}$ such that there exists a matrix-valued function $C:\mathbb{R}^n \rightarrow so(m)$ satisfying
$$[\tilde{Y}_0,\tilde{Y}_{\alpha}]=-\sum_{\beta=1}^m C_{\alpha}^{\beta}(x) \tilde{Y}_{\beta}$$
and we set
$$F^i_{\alpha}(x)=\tilde{Y}^i_{\alpha}(x), \qquad i=1,...,n $$
we find that $V=(Y,C,0)$ is a quasi-strong symmetry for the system \refeqn{equation_mechanics1}. \\
Indeed, the first determining equation \refeqn{determining_eq}    for $i=n+1,...,2n$ becomes:
\begin{eqnarray*}
[Y,\sigma]^i_{\alpha}&=&\sum_{j=1}^{2n} (Y^j \partial_j(\sigma^i_{\alpha})-\sigma^j_{\alpha} \partial_j(Y^i))\\
&=&\sum_{j=1}^n \tilde{Y}_0^j \partial_{x^j}(\tilde{Y}_{\alpha}^{i-n})-\sum_{j=n+1}^{2n} \tilde{Y}_{\alpha}^{j-n}\partial_{v^j}
\left(\sum_{k=1}^n \partial_{x^k}(\tilde{Y}_0^i) v^k\right)\\
&=&\sum_{j=1}^n \tilde{Y}_0^j \partial_{x^j}(\tilde{Y}_{\alpha}^{i-n})-\sum_{j=1}^n\tilde{Y}_{\alpha}^{j}\partial_{x^j}(\tilde{Y}_0^{i-n})\\
&=&[\tilde{Y},\tilde{Y}_{\alpha}]^{i-n}=-\sum_{\beta=1}^m C^{\beta}_{\alpha} \tilde{Y}^i_{\beta}=-\sum_{\beta=1}^m C^{\beta}_{\alpha} \sigma^i_{\beta}.
\end{eqnarray*}
Since $\sigma^i_{\alpha}=0$ for $i\leq n$ and $Y^i$ does not depend on $v$ for $i \leq n$, we have
$$[Y,\sigma]^i_{\alpha}=-\sum_{\beta=1}^m C^{\beta}_{\alpha} \sigma^i_{\beta}, \qquad i\leq n. $$
Furthermore
$$Y(\mu^i)-L(Y^i)=Y(\mu^i)-\mu(Y^i)=[Y,\mu]^i=0$$
because $Y$ is a symmetry of  \refeqn{equation_mechanics2}.\\

An interesting particular case within this class is given by the following equation
$$\frac{d^2X_t}{dt}=-\gamma \frac{dX_t}{dt},$$
representing (for $\gamma >0$)  the motion of a particle subjected  to  a linear dissipative force. This equation has the symmetry $\tilde{Y}=x\partial_x$,
so that the system
\begin{equation}\label{equation_mechanics5}
\left(\begin{array}{c}
dX_t\\
dV_t\end{array} \right)=\left( \begin{array}{c}
V_t\\
-\gamma V_t \end{array} \right)dt+\left(\begin{array}{c}
0\\
\alpha X_t\end{array}\right)dW_t,
\end{equation}
which provides the equation of a dissipative random harmonic oscillator, has the strong symmetry
$$ Y=\left(\begin{array}{c}
x\\
v\end{array}\right).$$
If we consider the local diffeomorphism
$$\Phi(x,v)=\left( \begin{array}{c}
\frac{1}{2}\log(x^2+v^2)\\
\frac{v}{x}\end{array}\right), $$
in the coordinates $(x',v')^T=\Phi(x,v)$ equation \refeqn{equation_mechanics5} becomes
\begin{eqnarray*}
dX'_t&=&\left(-\frac{\gamma (V'_t)^2 - V'_t}{1+V'^2_t}+\alpha^2 \frac{1-V'^2_t}{2+4V'^2_t+2V'^4_t}\right)dt+\frac{\alpha V'_t}{1+V'^2_t}dW_t\\
dV'_t&=&(-\gamma V'_t-V'^2_t)dt+\alpha dW_t.
\end{eqnarray*}
This system is not integrable but the equation for $V'_t$ is known in literature (see \cite{Gard1988}). Furthermore, as well as its deterministic counterpart,
this equation admits a superposition rule \cite{Ortega2009}.\\
Another interesting example within the class of the equations described by \refeqn{equation_mechanics1} is given by
 the following system
\begin{equation} \label{equation_mechanics6}
\left( \begin{array}{c}
dX_t\\
dZ_t\\
dV^x_t\\
dV^z_t
\end{array} \right)=\left(\begin{array}{c}
V^x_t\\
V^z_t\\
\frac{X_tf(\sqrt{X_t^2+Z_t^2})}{\sqrt{X^2_t+Z^2_t}}-\gamma V^x_t\\
\frac{Z_tf(\sqrt{X_t^2+Z_t^2})}{\sqrt{X^2_t+Z^2_t}}-\gamma V^z_t \end{array}\right) dt+ \left(\begin{array}{cc}
0 & 0\\
0 & 0\\
D & 0\\
0 & D \end{array} \right) \cdot \left(\begin{array}{c}
dW^1_t\\
dW^2_t
\end{array} \right),
\end{equation}
where $D \in \mathbb{R}_+$ and $f:\mathbb{R}_+ \rightarrow \mathbb{R}$ is a smooth function. The SDE \refeqn{equation_mechanics6} is a Langevin type
equation  describing a  point particle of unitary mass subjected to the central force $f(\sqrt{x^2+z^2})$, to an isotropic dissipation linear
in the velocities and to a space homogeneous random force. Since both the central force and the dissipation are invariant under the rotation group, the vector field
$$\tilde{Y}=\left(\begin{array}{c}
z\\
-x \end{array} \right)$$
is a symmetry of  \refeqn{equation_mechanics6} for $D=0$. Furthermore we have that
$$\left[\tilde{Y}, \left(\begin{array}{cc}
1 & 0 \\
0 & 1\end{array}\right)\right]=-\left(\begin{array}{cc}
1 & 0 \\
0 & 1\end{array}\right) \cdot \left(\begin{array}{cc}
0 & 1\\
-1 & 0
\end{array} \right).$$
So putting
\begin{eqnarray*}
Y&=&\left(\begin{array}{c}
z\\
-x\\
v^z\\
-v^x \end{array} \right)\\
C&=&\left(\begin{array}{cc}
0 & 1\\
-1 & 0
\end{array} \right)
\end{eqnarray*}
the infinitesimal stochastic transformation $V=(Y,C,0)$ is a quasi-strong symmetry for equation \refeqn{equation_mechanics6}. \\
In order to reduce  \refeqn{equation_mechanics6} using the  symmetry $V$, we have to find the  stochastic transformation $T=(\Phi,B,1)$ which puts $V$ in canonical form.
Solving the equations for $\Phi$ and $B$ we obtain
\begin{eqnarray*}
\Phi(x,z,v^x,v^z)&=&\left(\begin{array}{c}
\operatorname{acos}\left(\frac{x}{\sqrt{x^2+z^2}}\right)\\
\sqrt{z^2+x^2}\\
\frac{x v^z-z v^x}{x^2+z^2}\\
\frac{xv^x+zv^z}{\sqrt{x^2+z^2}} \end{array} \right)\\
B(x,z,v^x,v^z)&=&\left(\begin{array}{cc}
\frac{x}{\sqrt{x^2+x^2}} & \frac{z}{\sqrt{x^2+z^2}}\\
-\frac{z}{\sqrt{x^2+z^2}} & \frac{x}{\sqrt{x^2+z^2}}
\end{array} \right).
\end{eqnarray*}
With the new coordinates $\Phi=(\theta,r,v^{\theta},v^r)^T$ and with the new Brownian motion $dW'_t=B \cdot dW_t$ we have
\begin{eqnarray*}
d\Theta_t&=&V^{\theta}_t dt\\
dR_t&=&V^r_t dt\\
dV^{\theta}_t&=&\left(-\frac{2V^r_tV^{\theta}_t}{R_t}-\gamma V^{\theta}_t\right)dt+\frac{D}{R_t} dW'^2_t\\
dV^r_t&=&\left(R_t(V^{\theta}_t)^2-\gamma V^r_t+f(R_t) \right)dt+D dW'^1_t,
\end{eqnarray*}
where the solution $\Theta_t$ can be reconstructed from $(R_t,V^r_t,V^{\theta}_t)$.

\begin{remark}
For $\gamma=0$ equation \refeqn{equation_mechanics6} is a Lagrangian SDE with action functional
$$S=\int_0^t{\left(\frac{1}{2}((V^x_s)^2+(V^z_s)^2)-F(\sqrt{X^2_s+Z^2_s})\right)ds}+\int_0^t{X_sdW^1_s}+\int_0^t{Z_tdW^2_s},$$
where $F(r)=\int_0^r{f(\rho)d\rho}$. The flow of the quasi-strong symmetry $V$ leaves the functional $S$ invariant, but  equation \refeqn{equation_mechanics6} does not admit a conservation law associated with $V$.
Furthermore, as already noted in  \cite{Zung}, the reduction of \refeqn{equation_mechanics6} along $V$ allows us to reduce by one (and not
by two as in the deterministic case) the dimension of the system.
\end{remark}

\subsection{A financial mathematics application: the SABR model}
In this section we discuss  a stochastic volatility model used in mathematical finance to describe the stock price $s$ with volatility $u$
under an equivalent martingale measure for $s$ (see \cite{SABR1}). The deep geometric properties of this model, related with Brownian
motion on the Poincar\`e plane, are well known and suitably exploited in order to obtain asymptotic expansion formula for options evaluation  (see \cite{SABR2}).\\
The SABR model is a two-dimensional system of the  form
\begin{equation}\label{equation_SABR}
\left(\begin{array}{c}
dS_t \\
dU_t
\end{array}\right)=\left(\begin{array}{cc}
U_t (S_t)^{\beta} & 0\\
\alpha U_t \rho & \alpha U_t\sqrt{1-\rho^2}\end{array}\right) \cdot
\left(\begin{array}{c}
dW^1_t\\
dW^2_t
\end{array}\right),
\end{equation}
where $\beta,\alpha,\rho \in \mathbb{R}$ and $0<\beta<1$ and $0\leq \rho \leq 1$.\\
This SDE admits two  symmetries
\begin{eqnarray*}
V_1=(Y_1,C_1,\tau_1)&=&\left(\left(
\begin{array}{c}
0\\
1\end{array}\right),0,-\frac{2}{u}\right),\\
V_2=(Y_2,C_2,\tau_2)&=&\left(\left(
\begin{array}{c}
s\\
(1-\beta) u
\end{array}\right),0,0\right)\\
\end{eqnarray*}
and we can find a suitable random time change  transforming \eqref{equation_SABR} into an integrable SDE as a part of a stochastic transformation  $T=(\Phi,B,\eta)$ with $B=I_2$ satisfying
\begin{eqnarray*}
Y_1(\Phi)&=&\left(\begin{array}{c}
0\\
1\end{array}\right)\\
Y_2(\Phi)&=&\left(\begin{array}{c}
1\\
f(s,u)\end{array}\right)\\
Y_1(\eta)&=&-\tau_1 \eta\\
Y_2(\eta)&=&0
\end{eqnarray*}
(where $f$ is an arbitrary function). A solution to this system   is
\begin{eqnarray*}
\Phi(s,u)&=&\left(\begin{array}{c}
\log s\\
u \end{array}\right),\\
\eta&=&\frac{u^2}{s^{2-2\beta}}
\end{eqnarray*}
and putting  $t'=\int_0^t{\eta(S_s,U_s)ds}$ and $(s',u')=\Phi(s,u)$  we obtain the following  SDE in the new coordinates
$$\left(\begin{array}{c}
dS'_{t'} \\
dU'_{t'}
\end{array}\right)=\left(\begin{array}{c}
-\frac{1}{2} \\
0 \end{array}\right) dt' +\left(\begin{array}{cc}
1 & 0\\
\alpha e^{(1-\beta) S'_{t'}} \rho & \alpha e^{(1-\beta) S'_{t'}} \sqrt{1-\rho^2}\end{array}\right)
\left(\begin{array}{c}
dW'^1_{t'}\\
dW'^2_{t'}
\end{array}\right), $$
that is easily integrable.\\

In \cite{Antonov} the Authors, discussing the  non-correlated SABR model
($\rho=0$), propose a different  random time change
$$\tilde{t}=\int_0^t{U_s^2ds}$$
in order to derive  an analytic formula for the  solutions to \eqref{equation_SABR}.
According to this new time variable, the equation becomes
$$\left(\begin{array}{c}
dS_{\tilde{t}} \\
dU_{\tilde{t}}
\end{array}\right)=\left(\begin{array}{cc}
(S_{\tilde{t}})^{\beta} & 0\\
\alpha \rho & \alpha \sqrt{1-\rho^2}\end{array}\right)
\left(\begin{array}{c}
d\tilde{W}^1_{\tilde{t}}\\
d\tilde{W}^2_{\tilde{t}}
\end{array}\right),$$
and its symmetries are
\begin{eqnarray*}
\tilde{V}_1&=&\left(\left(
\begin{array}{c}
0\\
1\end{array}\right),0,0\right),\\
\tilde{V}_2&=&\left(\left(
\begin{array}{c}
s\\
(1-\beta) u
\end{array}\right),0,2(1-\beta)\right).\\
\end{eqnarray*}
Therefore the time change $\tilde{t}$ transforms $V_1$ into the strong symmetry $\tilde{V}_1$ and $V_2$ into the  symmetry  $\tilde{V}_2$, which,
since  $\beta$ is a constant, corresponds to a deterministic time change. The symmetry $\tilde{V}_2$, restricted to the $s$ variable, is the symmetry of a Bessel process:
indeed the process $S$ solves an equation for a spatial changed Bessel process. The Bessel process is one of the few one-dimensional stochastic processes
whose transition probability  is  explicitly known and  is a special case of the general  affine processes
 class ( see \cite{affine}). We remark that the time change $\tilde{t}$ can be uniquely characterized by the special form of $\tilde{V}_1$ and $\tilde{V}_2$,
 whose expression can be recovered within our symmetry analysis. Finally this last example suggests the  possibility of extending the integrability notion
 to processes which are not progressively reconstructible from gaussian processes but, more in general, from  other   processes  with notable analytical properties,
 such as Bessel process, affine processes or other processes.

\section*{Acknowledgements}
This work was  supported by National Group of Mathematical Physics (GNFM-INdAM).

\bibliographystyle{plain}
\bibliography{reduction_and_integrability(2)}

\end{document}